\documentclass[pdflatex,sn-mathphys-num]{sn-jnl}% Math and Physical Sciences Numbered Reference Style
\usepackage{graphicx}%
\usepackage{multirow}%
\usepackage{amsmath,amssymb,amsfonts}%
\usepackage{amsthm}%
\usepackage{mathrsfs}%
\usepackage[title]{appendix}%
\usepackage{xcolor}%
\usepackage{textcomp}%
\usepackage{manyfoot}%
\usepackage{booktabs}%
\usepackage{algorithm}%
\usepackage{algorithmicx}%
\usepackage{algpseudocode}%
\usepackage{listings}%
\theoremstyle{thmstyleone}%
\newtheorem{theorem}{Theorem}%  meant for continuous numbers
\theoremstyle{thmstyletwo}%

\theoremstyle{thmstylethree}%
\newtheorem{definition}{Definition}%

\begin{document}

\title[On the Converse of Shimoda's Theorem]{On the Converse of Shimoda's Theorem}


\author[1,2]{\fnm{Sevdiyor} \sur{Imomkulov}}\email{sevdiyor\_i@mail.ru}

\equalcont{These authors contributed equally to this work.}

\author*[3]{\fnm{Kamol} \sur{Rasulov}}\email{rasulov.kamol@gmail.com}

\affil[1]{\orgdiv{Department of Mathematical Analysis}, \orgname{National University of Uzbekistan named after Mirzo Ulugbek}, \orgaddress{\street{University st. 14}, \city{Tashkent}, \postcode{100174}, \country{Uzbekistan}}}

\affil[2]{\orgname{V.I. Romanovkiy Institute of Mathematics, Academy of Sciences of Uzbekistan}, \orgaddress{\street{University st. 9}, \city{Tashkent}, \postcode{100174}, \country{Uzbekistan}}}

\affil[3]{\orgdiv{Department of Economics and Digital Technologies}, \orgname{Tashkent University of Social Innovation}, \orgaddress{\street{Uzumzor st. 38}, \city{Tashkent}, \postcode{100015}, \country{Uzbekistan}}}


\abstract{We study Shimoda's theorem on the holomorphicity of a function $f(z,w)$ defined in a bidisc $U\times V=\left\{ \left| z \right|<1 \right\}\times \left\{ \left| w \right|<1 \right\},$ which is holomorphic with respect to $w\in V$ for each fixed $z\in U$ and is holomorphic with respect to $z\in U$ for each fixed $w\in E$, where $E\subset V$ is a countable set with at least one limit point in $V$. Shimoda proved that in this case $f(z,w)$ is holomorphic in $U\times V$ except for a nowhere dense closed subset of $U\times V$. We prove a converse: for any relatively closed nowhere dense set $S\subset U$ such that $\partial U\bigcup S$ is connected, there exists a function satisfying the conditions of Shimoda's theorem that is holomorphic in $(U\setminus S)\times V$ and doesn't extend holomorphically to any point of $S\times V$.}

\keywords{Hartogs' phenomenon, Shimoda's theorem, separately holomorphic functions, converse of Shimoda's theorem}


\pacs[MSC Classification]{32A05, 32A10}

\maketitle

\section{Introduction}\label{sec1}

The well-known Hartogs' Theorem \cite{bib1} states that if a function $f$ is holomorphic at any point of the domain $D\subset {{\mathbb{C}}^{n}}$ with respect to each of the variables ${{z}_{\nu }},$ then it is holomorphic in $D$ with respect to the set of all variables. This fundamental result laid the foundation for the theory of separately analytic functions. Hartogs' Theorem has different variations and generalizations in the works of several authors. Significant results in this direction have been obtained in the works of M.Hukuhara \cite{bib2}, I.Shimoda \cite{bib3}, T.Terada \cite{bib4}, J.Siciak \cite{bib5}, V.P.Zaharjuta \cite{bib6}, A.Sadullaev and E.M.Chirka \cite{bib7}, A.Sadullaev and S.Imomkulov \cite{bib8}, S.Imomkulov \cite{bib9,bib10}, A.Sadullaev and T.Tuichiev \cite{bib11}, M.Jarnicki and P.Pflug \cite{bib12,bib13} and A.Atamuratov \cite{bib14}.

The next question posed by Hukuhara was whether a function $f(z,w)$ is holomorphic or not in a domain $U\times V\subset {{\mathbb{C}}_{z}}\times {{\mathbb{C}}_{w}}$, if $f(z,w)$ is holomorphic with respect to $w$ for every fixed $z\in U$ and holomorphic with respect to $z$ only for ${{w}_{m}}\in V,$ where ${{w}_{m}}$ converges to an interior point ${{w}_{0}}\in V$.

In 1957, Shimoda \cite{bib3} gave the following solution to this problem.

\begin{theorem}[Shimoda's Theorem]
	Let $f(z,w)$ be a function defined in a polydisc $U\times V=\left\{ \left| z \right|<1 \right\}\times \left\{ \left| w \right|<1 \right\}\subset {{\mathbb{C}}^{2}}$ and $E\subset V$ be a countable subset having at least one limit point belonging to $V$. If
	
	1) for each fixed ${{z}^{0}}\in U,$ $f({{z}^{0}},w)\in \mathcal{O}(V);$
	
	2) for each fixed ${{w}^{0}}\in E,$ $f(z,{{w}^{0}})\in \mathcal{O}(U),$
	
\noindent then there exists a nowhere dense and relatively closed in $U$ subset $S\subset U$ such that $f\left( z,w \right)\in \mathcal{O}\left( (U\setminus S)\times V \right)$.
\end{theorem}

As shown in \cite{bib15}, the exceptional set $S$ (if $S\ne \varnothing $) cannot be compact; i.e., $S$ is relatively closed in $U$ and $\partial U\bigcup S$ is connected. Furthermore, the converse of Theorem 1 is proved in the case when $\partial U\bigcup S$ is path-connected.

\begin{definition}[see, for example, \cite{bib16}]
	A set $M$ is called connected if it cannot be represented as $M={{M}_{1}}\bigcup {{M}_{2}}$ so that both ${{M}_{1}}$ and ${{M}_{2}}$ are nonempty while the intersections ${{\overline{M}}_{1}}\bigcap {{M}_{2}}$ and ${{M}_{1}}\bigcap {{\overline{M}}_{2}}$ are empty.
\end{definition}

\begin{definition}[see, for example, \cite{bib16}]
	A set $M$ is called path-connected if every pair of points $a,b\in M$ can be joined by a path in $M$.
	
	(Here, by a path we mean a continuous mapping from the interval $\left[ \alpha ,\,\beta  \right]$ of the real axis into $\mathbb{C}$).
\end{definition}

\begin{theorem}[A.Atamuratov, K.Rasulov \cite{bib15}]
	Let $U\times V=\left\{ \left| z \right|<1 \right\}\times \left\{ \left| w \right|<1 \right\}$ be a unit polydisc in space ${{\mathbb{C}}^{2}}$. If $S\subset U$ is an arbitrary nowhere dense, relatively closed subset of $U$, such that the union $\partial U\bigcup S$ is path-connected and $E\subset V$ is a countable subset having at least one limit point belonging to $V$, then there exists a function $f\left( z,w \right)$ in $U\times V$, which satisfies the following conditions:
	
	1) for each fixed ${{z}^{0}}\in U,$ $f\left( {{z}^{0}},w \right)\in \mathcal{O}\left( V \right)$;
	
	2) for each fixed ${{w}^{0}}\in E,$ $f\left( z,{{w}^{0}} \right)\in \mathcal{O}\left( U \right)$;
	
	3) $f\left( z,w \right)\in \mathcal{O}\left( \left( U\setminus S \right)\times V \right)$ and $f\left( z,w \right)$ doesn't extend holomorphically to any point of the set $S\times V$.
\end{theorem}

To construct a function satisfying the conditions of Theorem 2, the authors used Mergelyan's theorem on polynomial approximation \cite{bib17}. In the general case, when the set $\partial U\bigcup S$ is connected, the proof of Theorem 2 cannot be applied without modification. For example, the following set

$$S=\left\{ z:\,\,\,\operatorname{Re}z=\frac{1}{2}(1-{{e}^{-\arg z}})\cos (\arg z),\,\,\operatorname{Im}z=\frac{1}{2}(1-{{e}^{-\arg z}})\sin (\arg z) \right\}\bigcup $$

$$\bigcup \left\{ z:\,\,\left| z \right|=\frac{1}{2} \right\}\bigcup \left\{ z:\,\,\frac{1}{2}\le \operatorname{Re}z<1,\,\,\operatorname{Im}z=0 \right\}$$

\noindent does not satisfy the conditions of Theorem 2.

\section{Main Result}\label{sec2}

Thus, we obtain the following final result.

\begin{theorem}
	Let $U\times V=\left\{ \left| z \right|<1 \right\}\times \left\{ \left| w \right|<1 \right\}\subset {{\mathbb{C}}^{2}}$ be a unit bidisc. If $S\subset U$ is an arbitrary nonempty, nowhere dense, relatively closed subset of $U$, such that the compact $K=S\bigcup \partial U$ is connected and $E\subset V$ is a countable subset having at least one limit point belonging to $V$, then there exists a function $f\left( z,w \right)$ in $U\times V,$ which satisfies the following conditions:
	
	1) for each fixed $z^0\in U$, $f(z^0,w)\in\mathcal O(V)$;
	
	2) for each fixed $w^0\in E$, $f(z,w^0)\in\mathcal O(U)$;
	
	3) $f(z,w)\in\mathcal O((U\setminus S)\times V)$ and $f(z,w)$ doesn't extend holomorphically to any point of the set $S\times V$.
\end{theorem}

\begin{proof} [Proof of Theorem 3.] In the case where the complement of the set $\overline{S}$ is connected, the proof of the theorem is given in \cite{bib15}.
	
We now consider the general case. Assume that the complement $\mathbb{C}\setminus \overline{S}$ consists of countably many open connected components, i.e.,  $\mathbb{C}\setminus \overline{S}={{\Omega }_{\infty }}\bigcup \left( \bigcup\limits_{\nu =1}^{+\infty }{{{\Omega }_{\nu }}} \right),$ where ${{\Omega }_{\infty }}$ is the unbounded component, and $\left\{ {{\Omega }_{\nu }} \right\}_{\nu =1}^{+\infty }$ are the bounded components, and ${{\Omega }_{\nu }}\ne \varnothing ,\,\,\nu =1,2,...\,.$ We take the following countable set of points: ${{\zeta }_{\infty }}\in {{\Omega }_{\infty }}$, ${{\zeta }_{\nu }}\in {{\Omega }_{\nu }},\,\nu =1,2,...\,.$ And we connect each point ${{\zeta }_{\nu }}$ to the point ${{\zeta }_{\infty }}$ by the segment $[{{\zeta }_{\infty }},{{\zeta }_{\nu }}]$. We denote the corresponding intersection by ${{L}_{\nu }}=[{{\zeta }_{\infty }},{{\zeta }_{\nu }}]\bigcap \overline{S}$.

We now set ${{S}^{(l)}}=\bigcup\limits_{\nu =1}^{l}{\partial {{\Omega }_{\nu }}}$. Clearly ${{S}^{(l)}}\subset {{S}^{(l+1)}}$. Suppose that the set $M=\overline{S}\setminus\bigcup\limits_{l=1}^{+\infty }{{{S}^{(l)}}}$ is nonempty.

We choose a sequence $\left\{ {{\delta }_{l}} \right\}_{l=1}^{+\infty }$ of positive numbers such that $0<{{\delta }_{l+1}}<{{\delta }_{l}}<1$ and ${{\delta }_{l}}\to 0$ as $l\to +\infty $. And we denote by

$$D({{L}_{\nu }},\delta _{l}^{\nu +m})=\left\{ z:\,\underset{\zeta \in {{L}_{\nu }}}{\mathop{\inf }}\,\left| z-\zeta  \right|<\delta _{l}^{\nu +m} \right\}$$

\noindent the $\delta _{l}^{\nu +m}$-neighborhood of the set ${{L}_{\nu }}$, such that $\partial {{\Omega }_{\nu }}\setminus \left( \bigcup\limits_{\nu =1}^{l}{D({{L}_{\nu }},\delta _{l}^{\nu +m})} \right)\ne \varnothing $, $\nu =1,2,...,l$.

Thus, we obtain a sequence of sets of the form

$$S_{m}^{(l)}=({{S}^{(l)}}\setminus \bigcup\limits_{\nu =1}^{l}{D({{L}_{\nu }},{{\delta }_{l}}^{\nu +m-1}))\bigcup \bigcup\limits_{\nu =1}^{l}{{{L}_{\nu }}},\,\,m=1,2,...}\,.$$

\noindent By construction, $S_{m}^{(l)}\ne \varnothing,$ $S_{m}^{(l)}\subset S_{m+1}^{(l)},$ $m=1,2,...,$ $\bigcup\limits_{m=1}^{+\infty }{S_{m}^{(l)}}={{S}^{(l)}}$ and the complement $\mathbb{C}\setminus S_{m}^{(l)}$ is connected.

We fix a sequence of positive numbers $\left\{ {{\varepsilon }_{ml}} \right\}$ such that ${{\varepsilon }_{ml}}>{{\varepsilon }_{(m+1)l}},\,\,{{\varepsilon }_{ml}}>{{\varepsilon }_{m(l+1)}},$ and ${{\varepsilon }_{ml}}\to 0$ as $m+l\to \infty $. We choose a corresponding collection of discs $B_{m1}^{(l)},B_{m2}^{(l)},...,B_{m{{j}_{ml}}}^{(l)}$ with radii not exceeding ${{\varepsilon }_{ml}}$ that cover the compact set $S_{m}^{(l)}$. We choose points $z_{mk}^{(l)}\in B_{mk}^{(l)}\setminus S_{m}^{(l)},\,\,k=1,2,...,{{j}_{ml}}$, such that

$$\bigcup\limits_{k=1}^{{{j}_{ml}}}{B_{mk}^{(l)}\supset \bigcup\limits_{k=1}^{{{j}_{(m+1)l}}}{\overline{B}_{(m+1)k}^{(l)}},\,\,z_{m1}^{(l)},z_{m2}^{(l)},...,z_{m{{j}_{m}}}^{(l)}\notin \overline{B}_{(m+1)k}^{(l)},\,\,\,k=1,2,...,{{j}_{(m+1)l}}\,.}$$

And so on continuing this process, we obtain a sequence of compact sets of the form  $K_{m}^{(l)}=\overline{U}\setminus \bigcup\limits_{k=1}^{{{j}_{ml}}}{B_{mk}^{(l)}}\,\,$ and a countable set of points $\left\{ z_{mk}^{(l)} \right\}$, such that the set of all limit points ${{\left\{ z_{mk}^{(l)} \right\}}^{\prime }}$ coincides with $\overline{S}.$

Moreover, for every point ${{z}^{0}}\in M$, there exists an index ${{m}_{l}}$ such that ${{z}^{0}}\in K_{m}^{(l)}$ for all $m\ge {{m}_{l}}$. Since, by construction $K_{m}^{(l)}\subset K_{m+1}^{(l)}$ and $M\subset \bigcup\limits_{m=1}^{\infty }{K_{m}^{(l)}}$ for every $l$, then there exists a strictly increasing sequence of integers ${{m}_{l}}:\,\,{{m}_{l}}<{{m}_{l+1}},\,\,l=1,2,...$, such that

\begin{equation}M\subset \bigcup\limits_{l=1}^{+\infty }{K_{{{m}_{l}}}^{(l)}}\,.\tag{1}\end{equation}

Let's put $G_{m}^{l}=K_{m}^{l}\bigcup S_{m}^{l}$. It is easy to see that the set $G_{m}^{(l)}\bigcup \left\{ z_{mk}^{(l)} \right\}_{k=1}^{{{j}_{ml}}}$ satisfies the conditions of Mergelyan's theorem (see \cite{bib17}).

By Mergelyan's theorem, for the compact set $G_{1}^{(1)}\bigcup \left\{ z_{1k}^{(1)} \right\}_{k=1}^{{{j}_{11}}}$, there exists a polynomial $p_{1}^{(1)}\left( z \right)$ such that $\left| p_{1}^{(1)}\left( z_{1k}^{(1)} \right) \right|\ge 2,\,\,\,k=1,2,...,{{j}_{11}}$ and ${{\left\| p_{1}^{(1)}\left( z \right) \right\|}_{G_{1}^{(1)}}}\le \frac{1}{3}$, where ${{\left\| p(z) \right\|}_{K}}=\underset{z\in K}{\mathop{\max }}\,\left| p(z) \right|$. Similarly, for compact sets $G_{2}^{(1)}\bigcup \left\{ z_{2k}^{(1)} \right\}_{k=1}^{{{j}_{21}}}$ and $G_{1}^{(2)}\bigcup \left\{ z_{1k}^{(2)} \right\}_{k=1}^{{{j}_{12}}}$ there exist polynomials $p_{2}^{(1)}\left( z \right)$ and $p_{1}^{(2)}\left( z \right)$, respectively, such that

$$\left| p_{2}^{(1)}\left( z_{2k}^{(1)} \right) \right|\ge {{A}_{3}}, \,\,\, k=1,2,...,{{j}_{21}},\,\,\, {{\left\| p_{2}^{(1)}\left( z \right) \right\|}_{G_{2}^{(1)}}}\le \frac{1}{3}$$

\noindent and

$$\left| p_{1}^{(2)}\left( z_{1k}^{(2)} \right) \right|\ge 2{{A}_{3}},\,\,\, k=1,2,...,{{j}_{12}},\,\,\, {{\left\| p_{1}^{(2)}\left( z \right) \right\|}_{G_{1}^{(2)}}}\le \frac{1}{3},$$

\noindent where ${{A}_{3}}={{\left\| p_{1}^{(1)}(z) \right\|}_{\overline{U}}}\,.$ Continuing this process, for each compact set $G_{m}^{(l)}\bigcup \left\{ z_{mk}^{(l)} \right\}_{k=1}^{{{j}_{ml}}}$, we obtain a polynomials $p_{m}^{(l)}\left( z \right):\,\,\left| p_{m}^{(l)}\left( z_{mk}^{(l)} \right) \right|\ge l{{A}_{m+l}},\,k=1,2,...,{{j}_{ml}}$ and ${{\left\| p_{m}^{(l)}\left( z \right) \right\|}_{G_{m}^{(l)}}}\le \frac{1}{3}$, where ${{A}_{m+l}}=\sum\limits_{\mu +\nu <m+l}{{{\left\| p_{\mu }^{(\nu )}(z) \right\|}_{\overline{U}}}}$.

Let $E\subset V$ be a polar compact. We denote by ${{t}_{m}}\left( w \right)$ Chebyshev polynomials (see \cite{bib18}) for the compact set $E$, all zeroes of which also lie on $E$. Since the capacity is $C\left( E \right)=0$, we have $\left\| {{t}_{m}}\left( w \right) \right\|_{E}^{\frac{1}{m}}\to 0$ as $m\to \infty $. This implies that there is a sequence of numbers ${{s}_{ml}}$ $\left( {{s}_{ml}}<{{s}_{(m+1)l}} \right)$ such that the following inequality holds

$$\left\| {{t}_{{{s}_{ml}}}}\left( w \right) \right\|_{E}^{\tfrac{1}{{{s}_{ml}}}}\le \frac{1}{2{{\left\| p_{m}^{(l)}(z) \right\|}_{\overline{U}}}}\,.$$

Let's put $P_{m}^{(l)}\left( z \right)={{\left[ p_{m}^{(l)}\left( z \right) \right]}^{{{s}_{ml}}}}$. Since all the roots of the polynomials ${{t}_{{{s}_{ml}}}}\left( w \right)$ lie on the unit disc $V$ the inequality ${{\left\| {{t}_{{{s}_{ml}}}}\left( w \right) \right\|}_{\overline{V}}}\le {{2}^{{{s}_{ml}}}}$ holds.

We consider the series

\begin{equation}f\left( z,w \right)=\sum\limits_{l=1}^{\infty }{P_{{{m}_{l}}}^{(l)}\left( z \right)\cdot {{t}_{{{s}_{{{m}_{l}}l}}}}\left( w \right)}\tag{2}\end{equation}

\noindent here, the sequences ${{m}_{l}}$ satisfies condition (1). The series (2) has the following properties:

1) for each fixed ${{z}^{0}}\in U,$

a) if ${{z}^{0}}\in S=\left( \bigcup\limits_{l=1}^{\infty }{{{S}^{(l)}}} \right)\bigcup M$ then there exists ${{l}_{0}}$ such that, for every $l>{{l}_{0}}$, ${{z}^{0}}\in S_{{{m}_{l}}}^{(l)}$ or ${{z}^{0}}\in K_{{{m}_{l}}}^{(l)}$. By the estimate

$$\left| P_{{{m}_{l}}}^{(l)}\left( {{z}^{0}} \right)\cdot {{t}_{{{s}_{{{m}_{l}}l}}}}\left( w \right) \right|\le \frac{1}{{{3}^{{{s}_{{{m}_{l}}l}}}}}\cdot {{2}^{{{s}_{{{m}_{l}}l}}}}={{\left( \frac{2}{3} \right)}^{{{s}_{{{m}_{l}}l}}}}$$

\noindent it follows from the decomposition

$$f\left( {{z}^{0}},w \right)=\sum\limits_{l=1}^{{{l}_{0}}}{P_{{{m}_{l}}}^{(l)}\left( {{z}^{0}} \right)\cdot {{t}_{{{s}_{{{m}_{l}}l}}}}\left( w \right)}+\sum\limits_{l={{l}_{0}}+1}^{\infty }{P_{{{m}_{l}}}^{(l)}\left( {{z}^{0}} \right)\cdot {{t}_{{{s}_{{{m}_{l}}l}}}}\left( w \right)},$$

\noindent that the series

\[f\left( {{z}^{0}},w \right)=\sum\limits_{l=1}^{\infty }{P_{{{m}_{l}}}^{(l)}\left( {{z}^{0}} \right)\cdot {{t}_{{{s}_{{{m}_{l}}l}}}}\left( w \right)}\]

\noindent converges uniformly in $V$, and consequently, $f\left( {{z}^{0}},w \right)\in \mathcal{O}\left( V \right),$

b) if ${{z}^{0}}\in U\setminus S,\,$ then there exists ${{l}_{0}}$ such that ${{z}^{0}}\in K_{{{m}_{l}}}^{(l)}$ for all $l>{{l}_{0}}$. Repeating the arguments used in the part a) we obtain $f\left( {{z}^{0}},w \right)\in \mathcal{O}\left( V \right).$

2) for each fixed ${{w}^{0}}\in E$, we have

$$\left| P_{{{m}_{l}}}^{(l)}\left( z \right) \right|\cdot \left| {{t}_{{{s}_{{{m}_{l}}l}}}}\left( {{w}^{0}} \right) \right|\le \left\| p_{m}^{(l)}(z) \right\|_{{\overline{U}}}^{{{s}_{{{m}_{l}}l}}}\cdot {{\left( \frac{1}{2{{\left\| p_{m}^{(l)}(z) \right\|}_{{\overline{U}}}}} \right)}^{{{s}_{{{m}_{l}}l}}}}={{\left( \frac{1}{2} \right)}^{{{s}_{{{m}_{l}}l}}}}$$

\noindent and therefore, we obtain $f\left( z,{{w}^{0}} \right)\in \mathcal{O}\left( U \right)$.

Thus, the function $f(z,w)$ defined by the series (2) satisfies conditions 1) and 2) of Theorem 3.

However, the following lower bound holds:

$$\left|f(z_{m_\nu\mu}^{(\nu)},w)\right|=\Biggl|P_{m_\nu}^{(\nu)}(z_{m_\nu\mu}^{(\nu)})t_{s_{m_\nu\nu}}(w)+\sum_{l=1}^{\nu-1}P_{m_l}^{(l)}(z_{m_\nu\mu}^{(\nu)})t_{s_{m_l l}}(w)+$$
	
$$+\sum_{l=\nu+1}^{\infty}P_{m_l}^{(l)}(z_{m_\nu\mu}^{(\nu)})t_{s_{m_l l}}(w)\Biggr|\ge\left|P_{m_\nu}^{(\nu)}(z_{m_\nu\mu}^{(\nu)})\right|\left|t_{s_{m_\nu\nu}}(w)\right|-\sum_{l=1}^{\nu-1}\left|P_{m_l}^{(l)}(z_{m_\nu\mu}^{(\nu)})\right|\left|t_{s_{m_l l}}(w)\right|-$$

$$-\sum_{l=\nu+1}^{\infty}\left|P_{m_l}^{(l)}(z_{m_\nu\mu}^{(\nu)})\right|\left|t_{s_{m_l l}}(w)\right|\ge(\nu A_{m_\nu+\nu})^{s_{m_\nu\nu}}\left|t_{s_{m_\nu\nu}}(w)\right|-$$

$$-\sum_{l=1}^{\nu-1}(A_{m_\nu+\nu})^{s_{m_l l}}\left|t_{s_{m_l l}}(w)\right|
-\sum_{l=\nu+1}^{\infty}\left(\frac23\right)^{s_{m_l l}},$$

\noindent $\mu =1,2,...,{{j}_{{{m}_{\nu }}\nu }}$, and when $w\notin E,$ then $d=dist\left( w,E \right)=\underset{\xi \in E}{\mathop{\inf }}\,\left| w-\xi  \right|>0$. Therefore, for every $\nu ,$ the following inequality holds 
$$\left| {{t}_{{{s}_{{{m}_{\nu }}\nu }}}}\left( w \right) \right|>{{d}^{{{s}_{{{m}_{\nu }}\nu }}}}.$$ 

\noindent  This implies that

$$\left|f(z_{m_\nu\mu}^{(\nu)},w)\right|\ge (\nu d A_{m_\nu+\nu})^{s_{m_\nu\nu}}-\nu(2A_{m_\nu+\nu})^{s_{m_\nu\nu}}-2\to+\infty,$$

\noindent $\nu \to \infty $, i.e., the function $f\left( z,w \right)$ is not holomorphic at any point of the set $S\times V.$

Now we prove that the function $f\left( z,w \right)$ is holomorphic in the domain $(U\setminus S)\times V$. Let ${{z}^{0}}\in U\setminus S,$ then we can show that there exists a neighborhood $U({{z}^{0}},r)$ such that the series (2) converges uniformly in the domain $U({{z}^{0}},r)\times V.$

If ${{z}^{0}}\in U\setminus S,$ then there exist a number ${{l}_{0}}$ and a neighborhood $U\left( {{z}^{0}},r \right)$ such that $U\left( {{z}^{0}},r \right)\subset K_{{{m}_{l}}}^{(l)}$ for all $l>{{l}_{0}}$. It follows that, for every $z\in U\left( {{z}^{0}},r \right)$ and for every $w\in V$, we have the representation

$$f\left( z,w \right)=\sum\limits_{l=1}^{{{l}_{0}}}{P_{{{m}_{l}}}^{(l)}\left( z \right)\cdot {{t}_{{{s}_{{{m}_{l}}l}}}}\left( w \right)}+\sum\limits_{l={{l}_{0}}+1}^{\infty }{P_{{{m}_{l}}}^{(l)}\left( z \right)\cdot {{t}_{{{s}_{{{m}_{l}}l}}}}\left( w \right)},$$

\noindent where the terms of the series with indices $l>{{l}_{0}}$ satisfy the following estimates:

$$\left| P_{{{m}_{l}}}^{(l)}\left( z \right)\cdot {{t}_{{{s}_{{{m}_{l}}l}}}}\left( w \right) \right|\le {{\left( \frac{2}{3} \right)}^{{{s}_{{{m}_{l}}l}}}}.$$

\noindent Consequently, $f\left( z,w \right)\in \mathcal{O}\left( U\left( {{z}^{0}},r \right)\times V \right)$.

Thus, the function $f\left( z,w \right)$ defined by the series (2) is holomorphic in the domain $\left( U\setminus S \right)\times V\,.$ This completes of proof. \end{proof}

It is clear that, in particular cases, when the set $M=\overline{S}\setminus\bigcup\limits_{l=1}^{+\infty }{{{S}^{(l)}}}$ is empty, or when the complement $\mathbb{C}\setminus \overline{S}$ has finitely many components, the proof of the theorem presents no particular difficulty.


\begin{thebibliography}{18}
	\bibitem{bib1} Hartogs F., Zur Theorie der analytischen Funktionen mehrerer
	unabh\"angiger Ver\"anderlichen, insbesondere \"uber die Darstellung derselben durch Reihen, welche nach	Potenzen einer Ver\"anderlichen fortschreiten. J. Math. Ann., 1906, 62: 1--88
	
	\bibitem{bib2}Hukuhara M., L'extensions du th\'eor\`eme d'Osgood et de Hartogs. Kans\^u-h\`oteisiki ogobi \^Oy\^o-Kaiseki, 1930, 48(8): 48--49
	
	\bibitem{bib3} Shimoda I., Notes on the functions of two complex variables. Journal of Gakugei, Tokushima University, Mathematics, 1957, 8(8): 1--3
	
	\bibitem{bib4} Terada T., Sur une certaine condition sous laquelle une fonction
	de plusieurs variables complexes est holomorphe. Diminution de la condition dans le th\'eor\`eme de Hartogs. Publications of the Research Institute for
		Mathematical Sciences,	1966, 2(3): 383--396 
	
	\bibitem{bib5} Siciak J., Separately analytic functions and envelopes of holomorphy of some lower dimensional subsets of ${{\mathbb{C}}^{n}}$. Annales Polonici Mathematici, 1969, 22(2): 145--171
	
	\bibitem{bib6} Zaharjuta V.P., Separately analytic functions, generalizations of Hartogs' theorem, and envelopes of holomorphy. Mathematics of the USSR-Sbornik, 1976, 30(1): 51--67 
	
	\bibitem{bib7} Sadullaev A.S., Chirka E.M., On continuation of functions with polar singularities. Mathematics of the USSR-Sbornik, 1988, 60(2): 377--384
	
	\bibitem{bib8} Sadullaev A.S., Imomkulov S.A., Extension of holomorphic and pluriharmonic functions with thin singularities on parallel sections. Proceedings of the Steklov Institute of Mathematics, 2006, 253(1): 144--159
	
	\bibitem{bib9} Imomkulov S.A., On holomorphic continuation of functions defined on a pencil of boundary complex lines. Izvestuya: Mathematics, 2005, 69(2): 345--363
	
	\bibitem{bib10} Imomkulov S.A., Continuation of $\mathbb{R}$-analytic functions parallel lines. TWMS J. Pure Appl. Math., 2023, 14(1): 120--124
	
	\bibitem{bib11} Sadullaev A.S., Tuichiev T.T., On continuation of Hartogs series that admit holomorphic extension to parallel sections. Uzbek mathematical journal, 2009, 1: 148--157 (in Russian)
	
	\bibitem{bib12} Jarnicki M., Pflug P., An extension theorem for separately holomorphic functions with analytic singularities. Annales Polonici Mathematici, 2003, 80: 143--161  
	
	\bibitem{bib13} Jarnicki M., Pflug P., Separately analytic functions. European Mathematical Society 2011
	
	\bibitem{bib14} Atamuratov A.A., On meromorphic continuation in a fixed direction. Mathematical Notes, 2009, 86(3): 323--327 
	
	\bibitem{bib15} Atamuratov A.A., Rasulov K.K., On Shimoda's theorem. Vestnik Udmurtskogo Universiteta. Matematika. Mekhanika. Komp'yuternye Nauki, 2023, 33(1): 17--31
	
	\bibitem{bib16} Shabat B.V., Introduction to complex analysis. Part I. Moscow: Nauka, 1976 (in Russian)
	
	\bibitem{bib17} Gamelin T.W., Uniform algebras. AMS, 1969
	
	\bibitem{bib18} Goluzin G.M., Geometricheskaya teoriya funktsii kompleknogo peremennogo (Geometric theory of functions of a complex variable). Moscow: Nauka, 1966
	
	
\end{thebibliography}
\end{document}